\documentclass{article}
\usepackage{amsmath,amssymb,amsthm}

\makeatletter
\renewcommand{\mod}[1]{\allowbreak \if@display \mkern 8mu \else 
\mkern 5mu\fi {\operator@font mod}\,\,#1}
\makeatother

\newcommand{\bc}{\mathbb C}
\newcommand{\bn}{\mathbb N}
\newcommand{\bq}{\mathbb Q}
\newcommand{\br}{\mathbb R}
\newcommand{\bz}{\mathbb Z}
\newcommand{\E}{\mathcal E}

\newcommand{\wf}{\widetilde{f}}
\newcommand{\wH}{\widetilde{H}}

\DeclareMathOperator{\Aut}{Aut}
\DeclareMathOperator{\Hom}{Hom}

\DeclareMathOperator{\rk}{rk}
\DeclareMathOperator{\ch}{ch}
\DeclareMathOperator{\td}{td}
\DeclareMathOperator{\Tyu}{Tyu}
\newtheorem{theorem}{Theorem}[section]

\newtheorem{corollary}[theorem]{Corollary}

\newtheorem{problem}[theorem]{Problem}

\begin{document}
\title{Self-correspondences of K3 surfaces \\
via moduli of sheaves and arithmetic hyperbolic reflection groups}
\date{}
\author{Viacheslav V. Nikulin\footnote{Supported by EPSRC grant
EP/D061997/1}} \maketitle

\begin{abstract}
In series of our papers with Carlo Madonna (2002--2008) we described 
self-correspondences of K3 surfaces over $\bc$ 
via moduli of sheaves with primitive isotropic Mukai vector for 
Picard number one or two of the K3 surfaces. 

Here we give a natural and functorial answer to   
the same problem for arbitrary Picard number. 

As an application, we characterise in terms of self-correspondences 
via moduli of sheaves K3 surfaces  with reflective Picard lattices, 
when the automorphism group of the lattice is generated by reflections 
up to finite index. It is known since 1981 that the number of 
reflective hyperbolic lattices is, in essential, finite. 
We also formulate several natural unsolved related problems.
\end{abstract}

\vskip1cm \centerline{\it To the memory of Vasily Alexeevich Iskovskikh} 
\vskip1cm

\section{Introduction}
\label{introduction}

In series of our papers with Carlo Madonna
\cite{Mad-Nik1} -- \cite{Mad-Nik3} and
\cite{Nik7} -- \cite{Nik9}, we considered self-correspondences of K3
surfaces (over $\bc$) via moduli of sheaves.

They are determined by primitive isotropic Mukai vectors
$v=(r,H,s)$ with $r \in \bn$, $s\in \bz$ and $H\in N(X)$ where
$N(X)$ is the Picard lattice of $X$, and $H^2=2rs$.
Due to Mukai \cite{Muk1}, \cite{Muk2} (see also Yoshioka \cite{Yoshioka1}),
the moduli space $Y=M_X(v)$ of stable (with respect to some ample element
$H^\prime\in N(X)$)  coherent sheaves on $X$ of rank $r$ with
first Chern class $H$ and Euler characteristic $r+s$ is again a
K3-surface. Chern class of the quasi-universal sheaf $\E$ on
$X\times Y$ gives some 2-dimensional algebraic cycle on $X\times Y$
and can be considered as a correspondence between $X$ and $Y$.
If $Y\cong X$, it can be considered as a self-correspondence of $X$
via moduli of sheaves with Mukai vector $v$.

In our papers above, we
studied when $Y\cong X$, and we gave a complete answer if the Picard
number of $X$ is 1 or 2, and $X$ is general for its
Picard lattice. We showed that all of them can be
reduced to so called Tyurin's isomorphisms $Y=M_X(v)\cong X$ when
$v=(\pm H^2/2,H,\pm 1)$ where $H\in N(X)$ has $\pm H^2>0$.

When $Y\cong X$, using Mukai's results,
in \cite{Nik9} we considered the action of
the self-correspondences of $X$ via moduli of sheaves on $H^2(X,\bq)$.
According to \cite{Nik9} and our considerations in this paper
(see Sec. \ref{corresp}), the action of Tyurin's isomorphism is
naturally given by the reflection $s_H$ with respect to $H$
(see \eqref{sH}), up
to the action of the group $\{\pm 1\}W^{(-2)}(N(X))$ where
$W^{(-2)}(N(X))$ is  generated by reflections in elements $\delta \in N(X)$
with $\delta^2=-2$ (they are
called $(-2)$ roots of $N(X)$). We call such
self-correspondence of $X$ as Tyurin's self-correspondence $Tyu(H)$ of $X$.

One can associate to $(-2)$ root $\delta \in N(X)$
a self-correspondences of $X$
which is  given by 2-dimensional cycle $\Delta+E\times E$ on $X\times X$
where $\Delta$ is the diagonal, and $E$ is an effective curve from
$|\pm \delta |$, and the action of this self-correspondence is given by the
reflection $s_\delta$ with respect to $\delta$.

Using well-known fact that the automorphism group of a rational quadratic 
form is generated by reflections, we obtain from here the main result of 
the paper: {\it any self-correspondence of $X$ defined by a primitive 
isotropic Mukai vector, and any their composition is numerically equivalent to 
composition of self-correspondences of $X$ defined by $(-2)$ roots, 
and Tyurin's self-correspondences up to finite index} (see Theorem 
\ref{thcorrcomptyurin} and Corollary \ref{corcorrcomptyurin} below for 
exact formulation). 

\medskip

It is also interesting to consider self-correspondences of $X$ with 
integral action in Picard lattice $N(X)$. One can consider them as 
analogous to automorphisms. See Sect. \ref{latt}.

We obtain that the action in $N(X)$ of self-correspondences of $X$
which are given by all $(-2)$ roots of $N(X)$ and by all
Tyurin's self correspondences $Tyu(H)$ with $H^2<0$ and
integral action in $N(X)$ (then $H\in N(X)$ is proportional to
a primitive root of $N(X)$) generate the group $\{\pm 1\}W(N(X))$
up to $\{\pm 1\}$. Here $W(N(X))$ is the reflection group of the
hyperbolic lattice $N(X)$; it is generated by reflections in all
roots of $N(X)$ with negative square. In particular, actions of
these self-correspondences in $N(X)$ generate a subgroup of finite
index in $O(N(X))$ if and only if $[O(N(X)):W(N(X))]<\infty$.
Such hyperbolic lattices $N(X)$ are called {\em reflective}.
Thus, we characterize K3 surfaces with reflective Picard lattices in 
terms of their self-correspondences.

Classification of reflective hyperbolic lattices is the subject
of the theory of arithmetic hyperbolic reflection groups.
In \cite{Nik3}, \cite{Nik4} and \cite{Vin1}, \cite{Vin2}, it was shown
that the number of similarity classes of reflective hyperbolic lattices
of rank at least 3 is finite.

Thus, for $\rk N(X)\ge 3$, action in $N(X)$ of
compositions of correspondences of $X$
defined by $(-2)$ roots of $X$ and by Tyurin's self-correspondences
$Tyu(H)$ with negative $H^2$ and integral action in $N(X)$ generate
a subgroup of infinite index in $O(N(X))$, if $N(X)$ is different from
finite number of similarity classes of hyperbolic lattices.

This results relate two different topics: Self-correspondences of K3
surfaces and Arithmetic hyperbolic reflection groups.

\section{Self-correspondences of K3 surfaces via moduli of sheaves}
\label{corresp}

In series of our papers \cite{Mad-Nik1} -- \cite{Mad-Nik3} (with 
Carlo Madonna) and \cite{Nik7} -- \cite{Nik9}, 
we considered self-correspondences of K3
surfaces (over $\bc$) via moduli of sheaves with primitive isotropic 
Mukai vectors. For Picard numbers one or two of K3 surface we described 
these self-correspondences in big details. 

Here we want to give a natural and functorial answer to this problem 
for arbitrary Picard number. Perhaps, the natural answer will be 
{\it to give some natural generators for these self-correspondences such 
that all other self-correspondences of this type will be compositions 
of these generators.} Below, we follow this idea.

We consider primitive isotropic Mukai vectors
$v=(r,H,s)$ where $r \in \bn$, $s\in \bz$ and $H\in N(X)$ where
$N(X)$ is the Picard lattice of $X$, and $H^2=2rs$.
Due to Mukai \cite{Muk1}, \cite{Muk2} (see also Yoshioka \cite{Yoshioka1}),
the moduli space $Y=M_X(v)$ of stable (with respect to some ample element
$H^\prime\in N(X)$)  coherent sheaves on $X$ of rank $r$ with
first Chern class $H$ and Euler characteristic $r+s$ is again a
K3-surface.

Write $\pi_X$, $\pi_Y$ for the projections of
$X\times Y$ to $X$ and $Y$. By Mukai \cite[Theorem 1.5]{Muk2}, the
algebraic cycle
\begin{equation}
Z_{\E}=(\pi^*_X \sqrt{\td_X})\cdot \ch(\E)\cdot
(\pi^*_Y\sqrt{\td_Y})/\sigma (\E)
\label{Mukaicycle}
\end{equation}
arising from the quasi-universal sheaf $\E$ on $X\times Y$ defines an
isomorphism of the full cohomology groups
\begin{equation}
\label{Mukaicycleisom}
f_{Z_\E}\colon H^*(X,\bq)\to H^* (Y,\bq),\quad
t\mapsto {\pi_{Y}}_*(Z_\E\cdot\pi^*_X t)
\end{equation}
with their Hodge structures (see \cite[Theorem 1.5]{Muk2}, for
details). Moreover, according to Mukai, it defines an isomorphism
of lattices (an isometry)
\begin{equation}
\label{Mukaiisom}
f_{Z_\E}\colon v^\perp \to H^4(Y,\bz)\oplus
H^2(Y,\bz)
\end{equation}
where $f_{Z_\E}(v)=w\in H^4(Y,\bz)$ is the fundamental cocycle, and the
orthogonal complement $v^\perp$ is taken in the Mukai lattice
$\wH(X,\bz)$ (with Mukai pairing $v^2=-2rs+H^2$). The Mukai's cycle
\eqref{Mukaicycle} and this construction can be viewed as a 
correspondence of $X$ and $Y$ via moduli of sheaves defined
by the primitive isotropic Mukai vector $v$. If $Y$ is isomorphic to $X$,
this defines a self-correspondence of $X$. 

Using this Mukai's
construction, in \cite{Nik9}, we considered the action of
the self-correspondence on $H^2(X,\bq)$. Below we recall this
construction. See \cite{Nik9} for details.

Composing $f_{Z_\E}$ with the projection $\pi\colon
H^4(Y,\bz)\oplus H^2(Y,\bz)\to H^2(Y,\bz)$ gives an embedding of
lattices
\begin{equation}
\label{Mukaiembed}
\pi\cdot f_{Z_\E}\colon H^\perp_{H^2(X,\bz)}\to
H^2(Y,\bz)
\end{equation}
that extends to an isometry
\begin{equation}
\label{isometryH2}
\wf_{Z_\E}\colon H^2(X,\bq)\to H^2(Y,\bq)
\end{equation}
of quadratic forms over $\bq$ by Witt's Theorem. If $H^2=0$, this
extension is unique.

If $H^2\not=0$, there are two such extensions, differing by the reflection 
\begin{equation}
\label{sH1}
s_H: x\to x- \frac{2(x\cdot H )H}{H^2} ,\ \ x\in H^2(X,\bq).
\end{equation}
where $s_H$ is identity on the 
orthogonal complement $H^\perp$, and $s_H(H)=-H$. 
We agree to take
\begin{equation}
\label{isometryH2h}
\wf_{Z_\E}(H)=(-r,0,s)\mod \bz v
\end{equation}
(see \cite{Nik9} for details). Another possibility is to consider both 
such extensions, their difference will be by the reflection $s_H$. 

The Hodge isometry \eqref{isometryH2} can be viewed as a minor
modification of Mukai's algebraic cycle \eqref{Mukaicycle} to obtain an
isometry in second cohomology. Clearly, it is defined by some
2-dimensional algebraic cycle on $X\times Y$,
because it only changes the Mukai isomorphism \eqref{Mukaicycleisom}
in the algebraic part.

Now let us assume that $Y=X$, that is the moduli of sheaves over $X$ are
parametrized by $X$ itself. Then the Hodge isometry \eqref{isometryH2}
defines the Hodge automorphism
\begin{equation}
\label{isometryH2X}
\wf_{Z_\E}\colon H^2(X,\bq)\to H^2(X,\bq)
\end{equation}
of quadratic forms over $\bq$ which gives an automorphism of the
transcendental periods $(T(X),H^{2,0}(X))$ of $X$.
Here $T(X)=(N(X))^\perp_{H^2(X,\bz)}$ is the transcendental lattice
of the K3 surface $X$. The automorphism \eqref{isometryH2X} is
called {\em the action of the self-correspondence of $X$ via moduli
of sheaves $\E$ with the primitive isotropic Mukai vector $v=(r,H,s)$}.
Changing the parametrization of sheaves by
an automorphism of $X$, one changes the action by the action of
the automorphism on $H^2(X,\bz)$ (and the correspondence by its composition
with the graph of the automorphism). Since non-trivial automorphisms of a K3
surface $X$ act non-trivially on $H^2(X,\bz)$, the exact choice of
the parametrization and the self-correspondence of $X$ via moduli of
sheaves with Mukai vector $v$ is defined by its action. We always choose
the simplest action that arrives at the most general K3 surfaces with
the given type of Mukai vector. See \cite{Nik9} and also
\cite{Mad-Nik1}--\cite{Mad-Nik3}, \cite{Nik7}, \cite{Nik8}
about this subject. We discuss this choice below for Tyurin's
isomorphisms.

We note that by Global Torelli Theorem for K3 surfaces \cite{PS},
the K3 surfaces $Y$ and $X$ are isomorphic if and only if two sublattices 
$H^2(X,\bz)$ and $\wf_{Z_\E}^{-1}(H^2(Y,\bz)$) for the defined action 
$\wf_{Z_\E}$ are conjugate by a rational automorphism $f\in O(H^2(X,\bq))$, 
that is $f^{-1}(H^2(X,\bz))=\wf_{Z_\E}^{-1}(H^2(Y,\bz))$, such that   
$(f\otimes \bc) (H^{2,0}(X))= H^{2,0}(X)$. Thus, the defined action 
$\wf_{Z_\E}$ shows when $Y\cong X$, and it is natural to consider.  

It is known that $Y\cong X$ for the Mukai vectors
\begin{equation}
\label{Tyurinv}
v=(\pm H^2/2, H,\pm 1)
\end{equation}
where $H\in N(X)$ and $\pm H^2>0$. We call this isomorphism as
{\em Tyurin's isomorphism}.
Tyurin \cite{Tyurin1} --- \cite{Tyurin3} described it geometrically
for general K3 surfaces $X$.
In \cite{Nik9} we calculated the action of the Tyurin's isomorphism,
and the result is that {\it there exists a unique choice of the identification 
$Y=M_X(v)=X$ such that the action \eqref{isometryH2X} satisfies} 
\begin{equation}
\label{Tyurinac}
\wf_{Z_\E}=s_H\mod \{\pm 1\}W^{(-2)}(N(X)). 
\end{equation}
Here 
$s_H$ is the reflection \eqref{sH1} with respect to $H$, that is $s_H$ 
is given by the formula  
\begin{equation}
\label{sH}
s_H: x\to x- \frac{2(x\cdot H )H}{H^2} ,\ \ x\in H^2(X,\bq),
\end{equation}
and $W^{(-2)}(N(X))\subset O(H^2(X,\bz))$ is generated by reflections 
$$
s_\delta:x\to x+(x\cdot \delta)\delta,\ \ x\in H^2(X,\bz),
$$
with respect to elements $\delta\in N(X)$ with 
square $\delta^2=-2$ (they are called reflections in $(-2)$ roots). 
  
More exactly, this is the action if $X$ is general with the
Mukai vector $v$, that is $N(X)\otimes \bq=\bq H$ and
$\Aut (T(X), H^{2,0}(X))=\pm 1$. Actually, one of $\pm s_H$ is
the only possible action if $X$ is general.
Thus, it coincides with Tyurin's geometric definition.
For an arbitrary $X$, we can take the action from the coset
$\{\pm 1\} W^{(-2)}(N(X))s_H$, and the choice is unique. Here we use Global
Torelli Theorem for K3 surfaces, \cite{PS}.
Geometrically this means that we choose the action which is
a specialization of the action from a general K3 surface with
the Mukai vector of the type $v$.
The self-correspondence of $X$ defined by moduli of sheaves on $X$
with Mukai vector $v$ of Tyurin's  type \eqref{Tyurinv}
and with action \eqref{Tyurinac} we call as
{\em Tyurin's self-correspondence $\Tyu(H)$ of X via moduli of sheaves
defined by an element $H\in N(X)$ with $H^2\not=0$.}

Assume that $\delta \in N(X)$ has $\delta^2=-2$.
By the Riemann--Roch theorem for K3 surfaces, $\pm \delta$ contains an
effective curve $E$. If $\Delta\subset X\times X$ is the diagonal, the
effective 2-dimensional algebraic cycle $\Delta+E\times E\subset X\times
X$ acts as the reflection $s_\delta$ in $H^2(X,\bz)$ (I learnt this from
Mukai \cite{Muk5}). We call this self-correspondences of $X$ as
{\em defined by $(-2)$ roots of $N(X)$.} Actions of
their compositions give the group $W^{(-2)}(N(X))$.

Using  \eqref{Tyurinac}, we obtain the following fundamental relation between
self-corres\-pondences of $X$ via moduli of sheaves with primitive 
isotropic Mukai vectors and Tyurin's
self-correspondences. 

We say that two correspondences of K3
surfaces $X$ and $Y_1$, and $X$ and $Y_2$ 
(that is 2-dimensional algebraic cycles on
$X\times Y_1$ and $X\times Y_2$) 
are {\em numerically equivalent} if their actions
$f_1:H^2(X,\bq)\to H^2(Y_1,\bq)$ and $f_2:H^2(X,\bq)\to H^2(Y_2,\bq)$ 
in second cohomology are equal: there exists an 
isomorphism $g:Y_1\to Y_2$ of K3 surfaces such that $f_1=g^\ast f_2$. 

We have the following statement which was indirectly formulated and 
proved in \cite{Nik9}. This is the main result of the paper. 

\begin{theorem}
\label{thcorrcomptyurin} Let $X$ be a K3 surface over $\bc$. It is 
known (see \cite{Nik1} and \cite{Nik5}) that the automorphism group 
$\Aut (T(X),H^{2,0}(X))$ of the transcendental periods of $X$ is 
a finite cyclic group of an order $n$ where $\phi(n)|\rk T(X)$ 
(here $\phi(n)$ is the Euler function). 
Then, obviously, there exists a finite number ($< n/2$) 
of primitive isotropic Mukai vectors $v_1,\dots,v_k$ giving 
self-correspondences of $X$ such that  
compositions of these self-correspondences give all possible actions  
on the transcendental periods $(T(X),H^{2,0}(X))$ up to $\pm 1$.  

Let $v=(r,H,s)$ be a primitive isotropic
Mukai vector on $X$. 
Then $Y=M_X(v)$ is isomorphic to $X$ if and only if the correspondence
between $X$ and $Y$ defined by $v$ with the action
$$
\wf_{Z_\E}\colon H^2(X,\bq)\to H^2(Y,\bq)
$$
given in \eqref{isometryH2} is numerically equivalent, up to $\pm
1$, to a composition of self-correspondences defined by $(-2)$ roots, 
Tyurin's self-correspondences, and self-correspondences defined by 
$v_1,\dots,v_k$. 

In particular, any self-correspondence of $X$ via moduli of sheaves 
with primitive isotropic Mukai vector and any their composition is 
numerically equivalent up to $\{\pm 1\}$ 
to composition of self-correspondences defined by $(-2)$ roots, 
Tyurin's self-correspondences and self-correspondences defined by 
$v_1,\dots ,v_k$.  
\end{theorem}

For a general (for its Picard lattice) K3 surface $X$ we, obviously, 
don't need self-correspondences defined by $v_1,\dots, v_k$, and 
we obtain the Corollary below. 

\begin{corollary}
\label{corcorrcomptyurin} Let $X$ be a K3 surface over $\bc$ which is 
general for its Picard lattice that is $\Aut (T(X),
H^{2,0}(X))=\{\pm 1\}$. Let $v=(r,H,s)$ be a primitive isotropic
Mukai vector on $X$.

Then $Y=M_X(v)$ is isomorphic to $X$ if and only if the correspondence
between $X$ and $Y$ defined by $v$ with the action
$$
\wf_{Z_\E}\colon H^2(X,\bq)\to H^2(Y,\bq)
$$
given in \eqref{isometryH2} is numerically equivalent, up to $\pm
1$, to a composition of correspondences defined by $(-2)$ roots, 
and Tyurin's self-correspondences.

In particular, any self-correspondence of $X$ via moduli of sheaves 
with primitive isotropic Mukai vector and any their composition is 
numerically equivalent up to $\{\pm 1\}$ 
to composition of self-correspondences defined by $(-2)$ roots, and  
Tyurin's self-correspondences.  
\end{corollary}

\begin{proof} Using $\wf_{Z_\E}$, let us identify $H^2(Y,\bz)$
with a sublattice $(\wf_{Z_\E})^{-1}(H^2(Y,\bz))$ in
$H^2(X,\bq)$. Then $\wf_{Z_\E}$ will be identified with the
identity of $H^2(X,\bq)$.

By the isomorphism \eqref{Mukaiisom} and the embedding
\eqref{Mukaiembed} of Hodge structures, we obtain an embedding of
the transcendental periods $(T(X),H^{2,0}(X))\subset (T(Y),$
$H^{2,0}(Y))$. If $Y\cong X$, we must then have
\begin{equation}
\label{isomtranper}
(T(X),H^{2,0}(X))=(T(Y),H^{2,0}(Y)).
\end{equation}
This identification of transcendental periods of $X$ and $Y$ is unique up to
multiplication by $\{\pm 1\}$ and by composition of self-correspondences 
defined by $v_1,\dots,v_k$. 

By Global Torelli Theorem for K3 surfaces \cite{PS}, $Y\cong X$ if
and only if the identification \eqref{isomtranper} can be extended
to an isomorphism of the sublattices $H^2(X,\bz)$ and $H^2(Y,\bz)$
of $H^2(X,\bq)$. Equivalently, there must exist $f\in
O(H^2(X,\bq))$ such that $f(H^2(X,\bz))=H^2(Y,\bz)$ and
$f|T(X)=\pm 1$. Changing $f$ to $-f$ if necessary, we can assume
that $f|T(X)=1$. Then $f$ is defined by its action in $N(X)\otimes
\bq$. By the well-known result about rational quadratic forms (e.g., see
\cite{Cassels}), such $f$ is a composition of reflections
$f=s_{H_m}\cdots s_{H_1}$ with respect to elements $H_i\in N(X)$
with $(H_i)^2\not=0$.  Note that these reflections give identity
in $T(X)$. Thus, our considerations are reversible.

If $Y=X$, we can consider $f$ as an identification of $H^2(X,\bz)$
with another copy of $H^2(X,\bz)$. Then the action is $f^{-1}=
s_{H_1}\cdots s_{H_m}$. It follows the statement.
\end{proof}

Theorem \ref{thcorrcomptyurin} and Corollary 
\ref{corcorrcomptyurin} show that $(-2)$ roots and Tyurin's 
self-correspondences of $X$ are fundamental for
all self-correspondences of $X$ via moduli of sheaves with primitive 
isotropic Mukai vector.   

The proof of these statements demonstrates an application of the purely 
arithmetic fact that the group of automorphisms of a rational quadratic 
form is generated by reflections. The same is valid for quadratic forms 
over fields of characteristic different from $2$.

These results can be also considered as some general (for any Picard number) 
alternative to our results with Carlo Madonna 
\cite{Mad-Nik1}--\cite{Mad-Nik3}, \cite{Nik6}--\cite{Nik8} about self-correspondences 
of K3 surfaces via moduli of sheaves with primitive isotropic Mukai 
vector which were valid for K3 surfaces with 
Picard number one or two only.

\begin{problem}
\label{probcorrcomptyurin}
Can one replace  in Theorem \ref{thcorrcomptyurin} and 
Corollary  \ref{corcorrcomptyurin} the  
numerical equivalence by some algebraic equivalence of
self-correspondences? 
\end{problem}

\section{K3 surfaces and arithmetic hyperbolic reflection groups}
\label{latt}

\subsection{Reflections and reflective hyperbolic lattices}
\label{refl}
We use the notation and terminology of \cite{Nik2} for lattices,
and their discriminant groups and forms. A {\em lattice} $L$ is a
nondegenerate integral symmetric bilinear form. That is, $L$ is a
free $\bz$-module of a finite rank with a symmetric pairing
$x\cdot y\in \bz$ for $x,\,y\in L$, assumed to be nondegenerate.
We write $x^2=x\cdot x$. The {\it signature} of $L$ is the
signature of the corresponding real form $L\otimes \br$. The
lattice $L$ is called {\em even} if $x^2$ is even for any $x\in
L$. Otherwise, $L$ is called {\em odd}. The {\it determinant} of
$L$ is defined to be $\det L=\det(e_i\cdot e_j)$ where $\{e_i\}$
is some basis of $L$. The lattice $L$ is {\em unimodular} if $\det
L=\pm 1$. The {\em dual lattice} of $L$ is
$L^*=\Hom(L,\,\bz)\subset L\otimes \bq$. The {\em discriminant
group} of $L$ is $A_L=L^*/L$; it has order $|\det L|$, and is
equipped with a {\em discriminant bilinear form} $b_L\colon
A_L\times A_L\to \bq/\bz$ and, if $L$ is even, with a {\em
discriminant quadratic form} $q_L\colon A_L\to \bq/2\bz$. To
define these, we extend the form on $L$ to a form on the dual
lattice $L^*$ with values in $\bq$.

An embedding $M\subset L$ of lattices is called {\em primitive} if
$L/M$ has no torsion. Similarly, a non-zero element $x\in L$ is
called {\em primitive} if $\bz x\subset L$ is a primitive
sublattice.

A non-zero element $\delta\in L$ is called positive, negative, and isotropic
if $\delta^2>0$, $\delta^2<0$, and $\delta^2=0$ respectively.
An element $\delta$ of a lattice $L$ is is called {\em root} if
$\delta^2\not=0$ and $\delta^2|2(\delta\cdot L)$.  For example, $\delta\in L$
with $\delta^2=\pm 2$ is root. It is called {\em $(\pm 2)$ root.}

Further $O(L)$ denotes the full automorphism group of the lattice $L$.
Each root $\delta\in L$ gives
{\em a reflection $s_\delta\in O(L)$ with respect
to $\delta$.} It is given by the formula
\begin{equation}
\label{sdelta}
s_\delta: x\to x- \frac{2(x\cdot \delta )\delta }{\delta^2} ,\ \ x\in L,
\end{equation}
and it is uniquely determined by the properties that
$s_\delta(\delta)=-\delta$ and $s_\delta$ gives identity on the orthogonal
complement $(\delta)^\perp_L$
to $\delta$ in $L$. Two proportional roots define the same reflection.
Thus, considering reflections $s_\delta$, we can restrict to primitive roots.
For a primitive $\delta\in L$ with $\delta^2\not=0$,
the formula \eqref{sdelta} defines an automorphism $s_\delta\in O(L)$ of $L$
if and only if $\delta$ is root. Otherwise, $s_\delta\in O(L\otimes \bq)$
is a rational automorphism of the rational quadratic form $L\otimes \bq$.
The reflection $s_\delta$ is called positive, negative if
$\delta$ is respectively positive, negative.

We denote by $W^{\pm}(L)$, $W^+(L)$, and $W(L)=W^-(L)$
normal subgroups of $O(L)$
generated by reflections in all positive and negative roots,
all positive roots, and all negative roots of $L$ respectively.

We denote by $W^{(-2)}(L)$ the normal subgroup of $O(L)$
generated by reflections in all elements of $L$ with square $-2$ that
is all $(-2)$ roots of $L$.

A lattice $L$ is called {\em hyperbolic} if it has signature $(1,\rho-1)$
where $\rho=\rk L$. A hyperbolic lattice $L$ is called {\em (classically)
reflective, equivalently $(-)$ reflective}
if $W^-(L)$ has finite index in $O(L)$,
that is the automorphism group
$O(L)$ is generated by all negative reflections of $L$, up to finite
index. Similarly, we call a hyperbolic lattice $L$ as
$(\pm)$ reflective, and $(+)$ reflective
if $W^{\pm}(L)$, and $W^+(L)$ has finite index in $O(L)$
respectively.

Further $L(m)$ denotes a lattice obtained by the
multiplication of the form of a lattice $L$ by positive $m\in \bq$.
The lattices $L$ and $L(m)$ are called {\em similar.} The automorphism
groups of similar lattices are naturally identified, and the lattices are
reflective of any type simultaneously.

Any hyperbolic lattice $L$ of rank one is reflective since $O(L)=\{\pm 1\}$.
Any hyperbolic lattice $L$ of rank two is reflective if and only if 
either $L$ has an isotropic element (then $O(L)$ is finite), 
or $L$ has a negative root (equivalently, a negative reflection). 
It is known since 1981 that there exist only finite number
of reflective hyperbolic lattices $L$ of $\rk L\ge 3$, up to similarity. 
We had shown this in \cite{Nik3}, \cite{Nik4}
for a fixed $\rk L\ge 3$. Vinberg \cite{Vin1}, \cite{Vin2} had shown that
$\rk L<30$ for reflective hyperbolic lattices $L$.

We note that recently similar finiteness results were completed
for reflective hyperbolic lattices over all
totally real algebraic number fields together. For example, see
\cite{Nik10} and references there.

Unfortunately, we don't know similar finiteness results for
$(\pm)$ reflective and $(+)$ reflective hyperbolic lattices.

\subsection{K3 surfaces with reflective Picard lattices}
\label{K3refl}

Further we consider Tyurin's self-correspondences $\Tyu(H)$ of $X$
with integral action in $N(X)$, that is $s_H|N(X)\in O(N(X))$.
By our discussion in Sec. \ref{refl}, this is equivalent for $H$
to be multiple of a primitive root in $N(X)$,
equivalently, $2(H\cdot N(X))H/H^2\subset N(X)$.
Self-correspondences of $X$ with integral action in $N(X)$
can be viewed as similar to ones defined by graphs of automorphisms of
$X$ and by $(-2)$ roots of $N(X)$.
We call self-correspondence $\Tyu(H)$ positive (respectively negative) if
$H^2>0$ (respectively $H^2<0$). From \eqref{Tyurinac} and 
our considerations in Section \ref{corresp}, we get the following result.

\begin{theorem} Let $X$ be a  K3 surface over $\bc$.

 The action of  compositions of all self-correspondences of $X$ defined by
$(-2)$ roots of $N(X)$ and by Tyurin's self-correspondences
$\Tyu(H)$ with integral action in $N(X)$
generate the group $\{\pm 1\}W^{\pm}(N(X))$ up to $\{\pm 1\}$.
They generate a subgroup of finite index in $O(N(X))$ if and only if
$N(X)$ is $(\pm)$ reflective.

The action of  compositions of all self-correspondences of $X$ defined by
$(-2)$ roots of $N(X)$ and by Tyurin's self-correspondences
$\Tyu(H)$ with negative $H^2$ and integral action in $N(X)$
generate the group
$$
\{\pm 1\}W(N(X))=\{\pm 1\}W^{-}(N(X))
$$
up to $\{\pm 1\}$. They generate a subgroup of finite index in
$O(N(X))$ if and only if $N(X)$ is (classically) reflective.
\label{thcorrefl}
\end{theorem}

By Piatetsky-Shapiro and Shafarevich \cite{PS} (this is an important
corollary of Global Torelli Theorem for K3 surfaces), we have
\begin{equation}
\label{PSShaf}
[O(N(X)):W^{(-2)}(N(X))\Aut X]<\infty,
\end{equation}
that is $W^{(-2)}(N(X))\Aut X$ is a subgroup of finite index in $O(N(X))$.
Here we identify $\Aut X$ with its action in $N(X)$. It has
a finite kernel. Moreover, $W^{(-2)}(N(X))\Aut X$ is the semi-direct
product of groups where $W^{(-2)}(N(X))$ is a normal subgroup.

Since by \cite{Nik3} and \cite{Nik4},
the number of similarity classes of reflective hyperbolic lattices
of a fixed rank at least 3 is finite, and $\rk N(X)\le 20$,
applying Theorem \ref{thcorrefl} and the result
\eqref{PSShaf} by Piatetsky-Shapiro and Shafarevich,
we obtain the following result. 

\begin{theorem} Let $X$ be a K3 surface over $\bc$.

By the result \eqref{PSShaf} by Piatetsky-Shapiro and Shafarevich \cite{PS},
the action in Picard lattice $N(X)$ of compositions of self-correspondences
of $X$ defined by all $(-2)$ roots of $N(X)$ and by graphs of automorphisms
of $X$ give a subgroup of finite index in $O(N(X))$.

In contrary, the action in Picard lattice $N(X)$ of compositions
of self-corres\-pondences of $X$ defined by all $(-2)$ roots of $N(X)$
and by all Tyurin's self-correspondences $Tyu(H)$ with $H^2<0$ and integral
action in $N(X)$ generate a subgroup of infinite index in $O(N(X))$ if and
only if $N(X)$ is not reflective. This is the case if
$\rk N(X)\ge 3$ and $N(X)$ is different from a finite number of
similarity classes of reflective hyperbolic lattices of rank $\ge 3$.

If $N(X)$ is reflective, then the action of
a finite number of $(-2)$ roots in $N(X)$ and a finite number of
Tyurin's self-correspondences $Tyu(H)$ with $H^2<0$ and integral
action in $N(X)$ generate
a subgroup of finite index of $O(N(X))$.
\label{mainref}
\end{theorem}

In the last statement, we use the well-known fact that the group
$W(N(X))$ is generated by a finite number of reflections in
negative roots of $N(X)$ if $N(X)$ is reflective.
This follows from the fact that the fundamental
chamber for $W(N(X))$ is a finite polyhedron in the hyperbolic space
determined by $N(X)$, the group $W(N(X))$ is generated by reflections in
negative roots which are perpendicular to codimension one faces of
this chamber. Their number is finite.

Unfortunately, we don't know finiteness for
$(\pm)$ reflective hyperbolic lattices of rank at least 3. This
raises an interesting question.

\begin{problem} Is the number of similarity classes of $(\pm)$
reflective hyperbolic lattices of rank at least 3 finite?

Here a hyperbolic lattice $S$ is $(\pm)$ reflective if the group $W^{\pm}(S)$
generated by reflections in  all positive and negative roots of
$S$ has finite index in $O(S)$.
\label{prplusminref}
\end{problem}

Theorem \ref{mainref} shows that to obtain a sufficient number
(with sufficiently arbitrary action in $N(X)$)
of self-correspondences of $X$ with integral action, one has to consider
all self-correspondences of $X$ generated by $(-2)$ roots,
Tyurin's self-correspondences $\Tyu(H)$ with integral action,
and graphs of automorphisms together.
This raises the following question which is interesting for
K3 surfaces with reflective and not reflective Picard lattices.

\begin{problem}
What is the kernel of the action in $N(X)$ of the group
of self-correspondences of $X$ generated by all $(-2)$ roots of $N(X)$,
$Tyu(H)$ with integral (or arbitrary)
action in $N(X)$, and graphs of automorphisms of $X$?
\label{prkernel}
\end{problem}


\

\

V.V. Nikulin \par Deptm. of Pure Mathem. The University of
Liverpool, Liverpool\par L69 3BX, UK; \vskip1pt Steklov
Mathematical Institute,\par ul. Gubkina 8, Moscow 117966, GSP-1,
Russia

vnikulin@liv.ac.uk \quad vvnikulin@list.ru


\begin{thebibliography}{ADSE}

\bibitem{Cassels} J.W.S. Cassels, \emph{Rational quadratic forms},
Academic Press, 1978.

\bibitem{Mad-Nik1} C. Madonna and V.V. Nikulin, \emph{On a classical
correspondence between K3 surfaces}, Proc. Steklov Inst. of Math.
\textbf{241} (2003), 120--153; (see also math.AG/0206158).

\bibitem{Mad-Nik2} C. Madonna and V.V. Nikulin, \emph{On a classical
correspondence between K3 surfaces II}, in:  M. Douglas, J.
Gauntlett, M. Gross  (eds.) Clay Mathematics Proceedings, Vol. 3
(Strings and Geometry), 2004, pp.285--300; (see also
math.AG/0304415).

\bibitem{Mad-Nik3}  C. Madonna and V.V. Nikulin,
\emph{Explicit correspondences of a K3 surface with itself}, Izvestiya:
Mathematics \textbf{72} (2008), no. 3, 497--508;
(see also math.AG/0605362, 0606239, 0606289).

\bibitem{Muk1} S. Mukai, \emph{Symplectic structure of the moduli space
of sheaves on an Abelian or K3 surface}, Inv. Math. \textbf{77}
(1984), 101--116.

\bibitem{Muk2} S. Mukai, \emph{On the moduli space of bundles on K3
surfaces I}, in: Vector bundles on algebraic varieties, Tata Inst.
Fund. Res. Studies in Math. no. \textbf{11} (1987), 341--413.

\bibitem{Muk3} S. Mukai, \emph{Duality of polarized K3 surfaces}, in: K. Hulek
(ed.) New trends in algebraic geometry. Selected papers presented
at the Euro conference, Warwick, UK, July 1996, Cambridge
University Press. London Math. Soc. Lect. Notes Ser. \textbf{264},
Cambridge, 1999, pp. 311--326.

\bibitem{Muk4} S. Mukai, \emph{Vector bundles on a K3 surface}, Proc.
ICM 2002 in Beijing, Vol. 3, pp. 495--502.

\bibitem{Muk5} S. Mukai,
\emph{Cycles on product of two K3 surfaces}, Lecture in the
University of Liverpool, February 2002.

\bibitem{Nik1} V.V. Nikulin, \emph{Finite automorphism groups of
K\"ahlerian surfaces of type K3}, Trans. Moscow Math. Soc.
\textbf{38} (1980), 71--135.

\bibitem{Nik2} V.V. Nikulin \emph{Integral symmetric bilinear forms and
some of their geometric applications}, Math. USSR Izv. \textbf{14}
(1980), no. 1, 103--167.


\bibitem{Nik3} V.V. Nikulin,
\emph{On arithmetic groups generated by reflections in Lobachevsky spaces},
Math. USSR Izv. \textbf{16} (1981), no. 3,
573--601.

\bibitem{Nik4} V.V. Nikulin, \emph{On the classification of arithmetic groups
generated by reflections in Lobachevsky spaces},
Math. USSR Izv. \textbf{18} (1982), no. 1, 99--123.

\bibitem{Nik5} V.V. Nikulin \emph{On the quotient groups of the
automorphism groups of hyperbolic forms by the subgroups generated
by 2-reflections. Algebraic-geometric applications}, J. Soviet
Math. {\bf 22} (1983), 1401--1476.

\bibitem{Nik6} V.V. Nikulin, \emph{On correspondences between K3
surfaces}, Math. USSR Izv. \textbf{30} (1988), no.2, 375--383.

\bibitem{Nik7} V.V. Nikulin, \emph{On correspondences of a K3 surface
with itself. I}, Proc. Steklov Inst. Math. \textbf{246} (2004),
204--226 (see also math.AG/0307355).

\bibitem{Nik8} V.V. Nikulin,  \emph{On Correspondences of a K3 surfaces
with itself. II}, in: JH. Keum and Sh. Kondo (eds.) Algebraic
Geometry. Korea-Japan Conf. in Honor of Dolgachev, 2004,
Contemporary mathematics \textbf{442}. AMS, 2007, pp. 121-- 172
(see also math.AG/0309348).

\bibitem{Nik9} V.V. Nikulin \emph{Self-correspondences of K3
surfaces via moduli of sheaves} In: Yuri Tschinkel and Yuri Zarhin (eds.)
Algebra, Arithmetic, and Geometry: Volume II: In Honor of Y.I. Manin,
(2008), Birkhauser (to appear) (see also math/0609233 [math.AG]).


\bibitem{Nik10} V.V. Nikulin,
\emph{On ground fields of arithmetic hyperbolic reflection groups.
III}, Preprint 2007, arXiv:0710.2340 [math.AG], 24 pages.

\bibitem{PS} I.I. Pjatetski\u i-\u Sapiro and I.R. \u Safarevi\u c,
\emph{A Torelli theorem for algebraic surfaces of type K3}, Math.
USSR Izv. \textbf{5} (1971), no. 3, 547--588.

\bibitem{Tyurin1} A.N. Tyurin, \emph{Cycles, curves and vector bundles
on algebraic surfaces}, Duke Math. J. \textbf{54} (1987), no. 1,
1--26.

\bibitem{Tyurin2} A.N. Tyurin, \emph{Special 0-cycles on a polarized K3
surface}, Math. USSR Izv. \textbf{30} (1988), no. 1, 123--143.

\bibitem{Tyurin3} A.N. Tyurin, \emph{Symplectic structures on the
varieties of moduli of vector bundles on algebraic surfaces with
$p_g>0$}, Math. USSR Izv. \textbf{33} (1989), no. 1, 139--177.

\bibitem{Vin1} \'E.B. Vinberg, \emph{The nonexistence of crystallographic
reflection groups in Lobachevski\v i spaces of large dimension},
Funct. Anal. Appl.  \textbf{15} (1981), 216--217.

\bibitem{Vin2} \'E.B. Vinberg, \emph{Absence of crystallographic
reflection groups in Lobachevski\v i spaces of large dimension},
Trans. Moscow Math. Soc. \textbf{47} (1985).

\bibitem{Yoshioka1} K. Yoshioka, \emph{Irreducibility of moduli spaces
of vector bundles on K3 surfaces}, Preprint math.AG/9907001, 21
pages.

\bibitem{Yoshioka2} K. Yoshioka, \emph{Some examples of Mukai's
reflections on K3 surfaces}, J. reine angew. Math. \textbf{515}
(1999), 97--123 (see also math.AG/9902105)

\end{thebibliography}
\end{document}